\documentclass[a4paper]{article}
\usepackage{amsmath, amsthm, amsfonts, amssymb, mathrsfs}
\usepackage[english]{babel}
\usepackage[all]{xy}
\usepackage{hyperref}
\usepackage{cite}
\usepackage{color}
\usepackage[utf8]{inputenc}
\usepackage[T1]{fontenc}

\theoremstyle{definition}
\newtheorem{maintheorem}{Theorem}

\newtheorem{fappl}{Theorem} 
 
\newtheorem{gk}{Theorem} 

\newtheorem{sappl}{Theorem} 
 
\newtheorem{TAppl}{Theorem} 

\newtheorem{Unfold}{Theorem} 

\newtheorem{FoAppl}{Theorem} 

\newtheorem{theorem}{Theorem}[section]
\newtheorem{proposition}[theorem]{Proposition}
\newtheorem{corollary}[theorem]{Corollary}
\newtheorem{lemma}[theorem]{Lemma}
\newtheorem{remark}[theorem]{Remark}
\newtheorem{definition}[theorem]{Definition}
\newtheorem{example}[theorem]{Example}
\newtheorem{problem}[theorem]{Problem}

\newcommand{\CC}{\mathbb{C}}
\newcommand{\PP}{\mathbb{P}}
\newcommand{\ZZ}{\mathbb{Z}}

\newcommand{\UUb}{\mathbb{U}}
\newcommand{\NN}{\mathbb{N}}
\newcommand{\OO}{\mathcal{O}}
\newcommand{\UU}{\mathcal{U}}
\newcommand{\FF}{\mathscr{F}}
\newcommand{\LL}{\mathcal{L}}

\newcommand{\KK}{\mathcal{K}}
\newcommand{\F}{\mathcal{F}}

\newcommand{\II}{\mathscr{I}}
\newcommand{\pp}{\mathfrak{p}}
\newcommand{\aff}{\mathfrak{aff}}

\newcommand{\sing}{\mathrm{Sing}}
\newcommand{\codim}{\mathrm{codim}}
\newcommand{\ann}{\mathrm{ann}}
\newcommand{\proj}{\text{Proj}}
\newcommand{\Hom}{\text{Hom}}

\title{On the geometry of the singular locus of a codimension one foliation in $\PP^n$}
\author{Omegar Calvo-Andrade\footnotemark[1], Ariel Molinuevo\thanks{The author was fully supported by CIMAT, Mexico.}, Federico Quallbrunn\thanks{The author was fully supported by Universidad de Buenos Aires, Argentina.}}
\date{}

\makeatletter
\def\blfootnote{\gdef\@thefnmark{}\@footnotetext}
\makeatother

\begin{document}

\maketitle

\bigskip

\begin{abstract}
We will work with codimension one holomorphic foliations over the complex projective space, represented by integrable forms $\omega\in H^0(\Omega^1_{\PP^n}(e))$. Our main result is that,  under suitable hypotheses, the Kupka set of the singular locus of $\omega\in H^0(\Omega^1_{\PP^3}(e))$, defined algebraically as a scheme, 
turns out to be arithmetically Cohen-Macaulay. As a consequence, we prove the connectedness of the Kupka set in $\PP^n$, and the splitting of the tangent sheaf of the foliation, provided that it is locally free. 
\end{abstract}

\blfootnote{\textup{2000} \textit{Mathematics Subject Classification}: 13D10, 14F17, 32S65}

\section{Introduction}

In this work, we consider sections $\omega\in H^0(\Omega^1_{\PP^n}(e)),\quad n\geq3$, such that its \textit{singular set}, denoted by $\sing(\omega)=\{p\in \PP^n :\  \omega(p)=0\}$ has codimension $\geq2$.
 
Such a section, represents a codimension one \textit{singular holomorphic distribution} of the complex projective space $\PP^n$, that is, at any regular point \emph{i.e.} $p\in \PP^n\setminus \sing(\omega)$, the differential 1-form $\omega(p)$ defines an hyperplane $T_{\omega}(p)=Ker(\omega(p))\subset T_{\PP^n}(p)$.

A distribution represented by $\omega\in H^0(\Omega^1_{\PP^n}(e))$ is \textit{involutive} if $\omega\wedge d\omega=0\in H^0(\Omega^3_{\PP^n}(2e))$. In this case, the Frobenius Theorem implies that $\omega$ is integrable and represents a codimension one \textit{regular holomorphic foliation} on $\PP^n\setminus \sing(\omega)$.   

We denote  respectively by 
\begin{eqnarray*}
\mathscr{D}(\PP^n,e)&=&\{\omega\in \PP(H^0(\Omega^1_{\PP^n}(e))):\  \codim(\sing(\omega))\geq 2\}  \\
\FF(\PP^n,e)&=&\{\omega\in \mathscr{D}(\PP^n,e):\  \omega\wedge d\omega=0\}
\end{eqnarray*}
the set of codimension one singular distributions  and singular foliations with normal class $e$ or degree $d=e-2$. The degree, or equivalently the normal class, is the main discrete invariant of a distribution or a foliation. These sets are not empty when $e\geq2$ or the degree $d\geq0$.

It is also known that $\mathscr{F}(\PP^n,e)$ is an algebraic subvariety of $\PP(H^0(\Omega^1_{\PP^n}(e)))$ and has several irreducible components if $n,e\geq3$.

Given $\omega\in \mathscr{D}(\PP^n,e)$ its singular set, $\sing(\omega)$, can be written as $\sing(\omega)=S_2(\omega)\cup\cdots\cup S_n(\omega)$, where $S_k(\omega)$ denotes the union of the irreducible components of codimension $k$.

A distribution may be seen in the following way. 
A section $\omega\in H^0(\Omega^1_{\PP^n}(e))$ may be regarded as a sheaf map $\omega:\OO_{\PP^n}\rightarrow\Omega^1(e)$. Taking duals, we get a cosection $\omega^\vee:(\Omega^1(e))^{\vee}\rightarrow\OO_{\PP^n}$, whose image is an ideal sheaf up to twist. After twisting by $\OO_{\PP^n}(e)$ we get $\mathscr{J}(e)=\mathcal{I}_{\sing(\omega)}(e)$. It induces the exact sequence 

\begin{equation}\label{ExactSeq}
0\rightarrow T_\omega\rightarrow T_{\PP^n}
\stackrel{\omega}{\rightarrow}\mathscr{J}(e)\rightarrow 0\qquad
T_\omega=ann(\omega)=\{X\in T_{\PP^n} :\  i_{X}\omega=0\},
\end{equation}
where $T_\omega$ is the \textit{tangent sheaf} of the distribution that $\omega$ represents. It is reflexive and in some cases, it is locally free. The quotient sheaf $\mathscr{J}(e)$, is the \textit{normal sheaf}. It is torsion free and its first Chern class is the normal class of the distribution.

The exact sequence (\ref{ExactSeq}) and its cohomology exact sequence provides the relations between the singular set as scheme and properties of the tangent sheaf that has implications on the global behaviour of the foliation.

From this point of view, the distribution is involutive if and only if, the tangent sheaf is closed under the Lie bracket operation of holomorphic vector fields $[T_\omega,T_\omega]\subset T_\omega$.    
 
These two points of view give rise to different notions of flat families, that coincides in some cases, see \cite{quallbrunn}.

\subsection{Distributions with locally free tangent sheaf}

As we have mentioned above, the tangent sheaf $T_\omega$ is reflexive but in general, it is not locally free. 
Foliations with locally free tangent sheaf, has been considered in \cite{omegar2} and \cite{fj}, from the point of view of deformations of foliations. Also in \cite{rl}, as in \cite{omegar2} and \cite{fj}, the authors obtain new irreducible components of the space of foliations coming from representations of $\aff(\CC)$, in this case as pullbacks of $\PP^2$ with weights.
In \cite{gir-pancoll} in dimension 3 and \cite{correa-etal} in any dimension, it is showed that $T_\omega$ is locally free if and only if $\sing(\omega)$ is pure of codimension 2. Even more so, the tangent sheaf splits as a sum of line bundles if and only if the singular set is \textit{arithmetically Cohen--Macaulay} (aCM).  

On the other hand, given a rank $(n-1)$ holomorphic vector bundle $E$, for a suitable integer $k\gg0$, a generic linear map $\varphi:F=E(k)\rightarrow T_{\PP^n}$ defines a codimension one distribution \cite[Appendix]{Jardim-etal}. It is an interesting question 
to determine if the set $\mathcal{H}om(F,T_{\PP^n})$ contains integrable maps.
\[
\mathscr{D}: 0\rightarrow F\stackrel{\varphi}{\rightarrow}
T_{\PP^n} \rightarrow J\rightarrow0\quad 
\mbox{ with }\quad[\varphi(F),\varphi(F)]\subset \varphi(F). 
\]      

Until now, the only integrable known examples are those where the tangent sheaf splits as a direct sum of line bundles, i. e. $T_\omega\simeq \oplus \OO_{\PP^n}(\alpha_i).$ Moreover, since the bundle $T_{\PP^n}$ is stable, we have $1\geq\alpha_i\in\ZZ$ for all $i=1,\dots,n-1.$

\

In \cite{omegar2} the following problem is stated that we will address 
in the present work. See Theorem \ref{SplitTan}, for a partial solution.

\begin{problem} If the tangent sheaf of a foliation $T_\omega$ is locally free. Is it true that $T_\omega$ splits? 
\end{problem} 

Split tangent sheaf foliations are represented by homogeneous differential 1-forms in $\CC^{n+1}$ of the form
\[
\omega=i_{R}i_{X_1}\dots i_{X_{n-1}}dz_0\wedge\cdots\wedge dz_n
\]
where $R=\sum_{i=0}^n x_i\frac{\partial }{\partial x_i}$ denotes the radial vector field on $\CC^{n+1}$ and 
$X_j$ are homogeneous vector fields in $\CC^{n+1}$ such that 
\[
[X_i,X_j]=\sum_{\ell=1}^{n-1} C^{\ell}_{ij}\cdot X_{\ell},\quad \mbox{and}\quad [X_i,R]=(1-deg(X_i))X_i,
\]
and the foliation is defined by a representation of a Lie Algebra $\mathfrak{g}$ in homogeneous vector fields, see \cite{fj}.

A second problem that we consider in this note 
is the following one, stated in \cite{cerveau}. See Corollary \ref{K-connected} for a partial solution.

\begin{problem}
Let $\omega\in\FF(3,e)$ be a foliation. Is the set $S_2(\omega)$ connected?
\end{problem}

Observe that it is true in $n\geq4$ but we can find non--integrable holomorphic distributions with locally free tangent sheaf and singular disconnected set \cite{Jardim-etal}.

\subsection{First order unfoldings}

The theory of unfoldings for integrable differential forms was developed by Tatsuo Suwa in \cite{suwa-unfoldings}.
It was originally inspired by the work of Mather and Wasserman on unfoldings of holomorphic function germs, so most of the results in \cite{suwa-unfoldings, suwa-multiform, suwa-meromorphic} are proven for germs of holomorphic integrable $1$-forms. We recall the main definitions used by Suwa on those works.  
 Let us denote as $\OO_{\CC^{n+1},p}$ and $\Omega^1_{\CC^{n+1},p}$ the analytic germs of functions and differential 1-forms around $p\in\CC^{n+1}$, respectively. If $\varpi\in \Omega^1_{\CC^{n+1},p}$ defines a foliation, the space of first order unfoldings of $\varpi$ can be parameterized as
\begin{equation*}
U_p(\varpi) = \left\{(h,\eta)\in \OO_{\CC^{n+1},p}\times\Omega^1_{\CC^{n+1},p}:\ h\ d\varpi= \varpi\wedge (\eta - dh)\right\}\big/\CC.(0,\varpi).
\end{equation*}
For a generic $\varpi$, the projection of $U_p(\varpi)$ to the first coordinate defines an ideal $I_p(\varpi)\subseteq \OO_{\CC^{n+1},p}$. This ideal gives a good algebraic structure to study $U_p(\varpi)$ and was used by Suwa to classify first order unfoldings of rational and logarithmic foliations, see \cite{suwa-meromorphic,suwa-multiform}. We refer the reader to \cite{suwa-review} for a review of several results on unfoldings of foliations, together with several applications of unfoldings to rigidity and finite determinacy of codimension $1$ foliations.

\medskip

In \cite{moli} a study of unfolding of \emph{global} foliations is, to our best knowledge, first carried out. In that work the first order unfoldings of logarithmic and rational foliations on $\PP^n$ are computed. In \cite{mmq} more general cases are studied and, for example, the unfoldings of foliations with split tangent sheaf are determined. In order to study the unfoldings of foliations on projective space some modifications are made to the objects originally defined by Suwa. Such modifications will be used by us in this work so we will now recall them.

\medskip 

For $\omega\in \FF(\PP^n,e)$, first order unfoldings can be parameterized in a way analogous to that of Suwa as
\begin{equation*}
\begin{aligned}
 &U(\omega)=\\
 &=\left\{(h,\eta)\in H^0(\PP^n,\OO_{\PP^n}(e))\times H^0(\PP^n,\Omega^1_{\PP^n}(e))\,\colon\, h\ d\omega  = \omega\wedge (\eta - dh)\right\}\big/\CC.(0,\omega).
\end{aligned}
\end{equation*}
Since $U(\omega)$ is a finite dimensional vector space there is no ideal associated to it. To remedy this shortcoming one can proceed as follows. Let $S=\CC[x_0,\ldots,x_n]$ be the ring of homogeneous coordinates in $\PP^n$ and consider $\omega$ as an affine differential form in $\CC^{n+1}$, the cone of $\PP^n$. Then we recall from \cite{moli} the $S$-module of \emph{graded projective unfoldings},
\[
\UUb(\omega) = \left\{(h,\eta)\in S\times \Omega^1_{S}: \ L_R(h)\ d\omega = L_R(\omega)\wedge(\eta - dh) \right\}\big/ S.(0,\omega),
\]
where $L_R$ is the Lie derivative with respect to the radial vector field \linebreak $R=\sum_{i=0}^nx_i\frac{\partial}{\partial x_i}$, see \cite[Definition 3.1, p.~1598]{moli}.

\smallskip

The projection of $\UUb(\omega)$ to the first coordinate defines an ideal $I(\omega)\subseteq S$ emulating the situation in the local analytic setting. We will call $I(\omega)$ the \emph{ideal of graded projective unfoldings} of $\omega$, or simply, the \emph{ideal of unfoldings of $\omega$} if no confusion can arise.

\subsection{Statement of the results}\label{results}

Let $\omega\in \FF(\PP^n,e)$ be a foliation. The singular locus of $\omega$, can be decomposed as the union of
\[
 \sing(\omega)= \overline{\KK_{set}}\cup\LL_{set},
\]
where $\KK_{set}$ is the Kupka set 
\[
\KK_{set}(\omega)=\{ p\in \sing(\omega) :\  d\omega(p)\neq0\}\subset S_2(\omega).
\]
and $\LL_{set}$ is the non-Kupka set 
\[
\LL_{set}(\omega)=\overline{Sing(\omega)\backslash \overline{\KK_{set}}}.
\]
Recall that $\LL_{set}(\omega)\supset S_k(\omega),\ k\geq3$, but in some cases it also has codimension two components.

We need to introduce some refinements of these sets and consider them as schemes $\KK$ and $\LL$, see Definitions \ref{Kscheme} and \ref{NKscheme}, we refer the reader to \cite{mmq} for a complete exposition on the subject. 

\

%
%

Now we are able to state our main result.

\begin{maintheorem}\label{teo4-intro}
Let $\omega\in\FF(\PP^3,e)$ such that $Sing(\omega)$ is reduced and $\KK\cap\LL =\emptyset$, then $\KK$ is \emph{arithmetically Cohen-Macaulay} (aCM).
\end{maintheorem}

%
%
%

As a first application of our main result, we get the Theorem \ref{SplitTan} and Theorem \ref{gk-sing}, that gives a criterion for $T_\omega$ to be locally free and a direct sum of line bundles (for short split).

 \

Then, in Theorem \ref{K-connected}, we show that $\KK$ is connected.

\

We also apply our results in the case of foliations with $\KK_{set}$ compact and provide an algebraic proof of the fact that those foliations have a meromorphic first integral, see Theorem \ref{complete-intersection}.

\

Finally, in Section \ref{aci} we state two criterions for integrability of a differential form $\omega\in H^0(\Omega^1_{\PP^n}(e))$: Theorem \ref{teo5}, in terms of the ideals of the singular locus and the unfolding ideal of $\omega$; and then Theorem \ref{teo6}, which states that if $\sing(\omega)$ is smooth and of codimension 2 then $\omega$ is a degree 0 rational foliation or it is not integrable.

\section{Preliminaries}

\subsection{Integrable $1$-forms in $\PP^n$}

Let us first recall some results and definitions from \cite{mmq} that we are going to use, specifically Theorem \ref{teo1} and Theorem \ref{technical1} below.

From Theorem \ref{teo1} we have the equality of the radicals of the Kupka ideal $K(\omega)$ and the unfoldings ideal $I(\omega)$, see \cite[Theorem 4.12, p.~13]{mmq}. From Theorem \ref{technical1} we get the equality of the ideal $K(\omega)$ and of the ideal $I(\omega)$, see \cite[Corollary 4.15, p.~14]{mmq}.

We would like to state here that we will always consider $\PP^n$ with $n\geq 3$ and that, unless we mentioned it, all cohomologies are going to be on $\PP^n$.

\bigskip

We consider a section $\omega\in H^0(\PP^n,\Omega^1_{\PP^n}(e))$ as a  
1--form in $\CC^{n+1}$ 
\[
\omega=A_0(\mathbf{x})dx_0+\cdots + A_n(\mathbf{x})dx_n\quad \mathbf{x}=(x_0,\dots,x_n)
\] 
where $A_0,\dots A_n$ are homogeneous of degree $e-1$ and satisfying 
\[
i_{R}\omega=\sum_{j=0}^n x_j\cdot A_j(\mathbf{x})=0,\quad R=\sum_{j=0}^n x_j\frac{\partial}{\partial x_j}\ .
\]

\begin{definition} We define the graded ideals of $S$ associated to $\omega\in H^0(\Omega^1_{\PP^n}(e))$ as
\begin{align*}
I(\omega) &:= \left\{ h\in S:\  h\ d\omega = \omega\wedge\eta\text{ for some } \eta\in\Omega^1_S\right\}\\
J(\omega) &:= \left\{ i_X(\omega)\in S:\ X\in T_S \right\}.
\end{align*}
$I(\omega)$ and $J(\omega)$ are the ideal of \emph{first order unfoldings} and the ideal of \emph{trivial first order unfoldings}. We will also denote them $I=I(\omega)$ and $J=J(\omega)$ if no confusion arises.
\end{definition}

\begin{remark}\label{1notinI} 
 Note that by contracting with the vector fields $\frac{\partial}{\partial x_i}$, for $i=0,\ldots,n$,  we get that $J(\omega)$ defines the singular locus of $\omega$.

 Also, notice that, when $\omega$ is integrable, by contracting the integrability condition by a vector field $X$, one can see that $J(\omega)\subset I(\omega)$. This implies that the variety defined by the ideal $I(\omega)$ has codimension $\geq 2$.
 \end{remark}

Now, we introduce the Kupka and the non--Kupka schemes.

\begin{definition}\label{Kscheme} For $\omega\in\FF(\PP^n,e)$, we define the \emph{Kupka scheme} $\KK(\omega)$ as the scheme theoretic support of $d\omega$ at $\Omega^2_{S}\otimes_S S\big/J(\omega)$. Then, $\KK(\omega)=\proj(S/K(\omega))$ where $K(\omega)$ is the homogeneous ideal defined as
\[
K(\omega)=\ann(\overline{d\omega})+J(\omega)\subseteq S,\quad \overline{d\omega}\in \Omega^2_{S}\otimes_S S\big/J(\omega).
\]
We will denote $\KK=\KK(\omega)$ and $K=K(\omega)$ if no confusion arises.
\end{definition}

One could also define $K(\omega)$ as $K(\omega)=(J(\omega)\cdot \Omega^2_S:d\omega)$. Then, given that $\Omega^2_S$ is free, we can also write
\begin{equation}\label{Kbis}
K(\omega)=(J(\omega):\II(d\omega)).
\end{equation}
where we are denoting as $\II(d\omega)$ to the ideal generated by the polynomial coefficients of $d\omega$.

In the case of two ideals $J,K\subseteq S$, we define the \emph{saturation} of $J$ with respect to $K$ as
\[
\left(J:K^\infty\right) := \bigcup_{d\geq 1} \big(J:K^d\big).
\]

\begin{definition}\label{NKscheme} For $\omega\in\FF(\PP^n,e)$, we define the \emph{non-Kupka scheme} $\LL(\omega)$ as the projective scheme $\proj(S/L(\omega))$, where $L(\omega)$ is the homogeneous ideal defined by
\[
L(\omega)=(J(\omega):K(\omega)^{\infty}).
\]
We will write $\LL=\LL(\omega)$ and $L=L(\omega)$ if no confusion arises.
\end{definition}

Let $\mathfrak{p}$ be a point in $\PP^n$, \emph{i.e.}, an homogeneous prime ideal in $S$ different from the \emph{irrelevant ideal} $(x_0.\ldots,x_n)$.

\medskip

\begin{definition} We say that $\mathfrak{p}\in\PP^n$ is a \emph{division point of $\omega$} if $1\in I(\omega)_\mathfrak{p}$.
\end{definition}
Notice that the definition of division point is simply saying that, for such a point we have that, locally around $\mathfrak{p}$ there is a $1$-form $\eta\in \Omega^1_{\PP^n,\mathfrak{p}}$ such that $d\omega_\mathfrak{p}=\omega_\mathfrak{p}\wedge\eta$, so the germ of $\omega$ on $\mathfrak{p}$ ''divides" the germ of $d\omega$, hence the name. 

\bigskip

We now define a subset of the moduli space of foliations. It is the subset of foliations in which the ideal $I(\omega)$ has a well-behaved relation with the ideal $K(\omega)$. It is worth mentioning that all the foliations that we know of belong to this subset. 

\begin{definition}\label{generic} We define the set $\UU\subseteq\FF(\PP^n,e)$ as
\[
\UU = \left\{\omega\in\FF(\PP^n,e): \ \forall \mathfrak{p}\not\in\KK(\omega),\,\mathfrak{p}\text{ is a division point of }\omega\right\}.
\]
\end{definition}

\begin{theorem}\label{teo1}
Let $\omega\in\mathcal{U}\subseteq\FF(\PP^n,e)$. Then,
\[
\sqrt{I}=\sqrt{K}.
\]
Even more so, if $\sqrt{I}=\sqrt{K}$ then $\omega\in\UU$.
\end{theorem}

\begin{theorem}\label{technical1}
Let $\omega\in\FF(\PP^n,e)$ be such that $J=\sqrt{J}$ and $\KK\cap \LL= \emptyset$. Then
$I = K$.
\end{theorem}

\subsection{Arithmetically Cohen-Macaulay subschemes}

Arithmetically Cohen-Macaulay (aCM for short) subschemes of projective spaces $Y\subseteq \PP^n$ are those whose \emph{homogeneous coordinate ring} $\Gamma_*(\OO_Y)=\oplus_\ell H^0(\OO_Y(\ell))$ is Cohen-Macaulay. Notice that this definition depends on the immersion of $Y$ on $\PP^n$. So, although it implies that the local rings $\OO_{Y,p}$ are Cohen-Macaulay for each $p\in Y$, it is not an inherent property of $Y$ but rather a property about the form $Y$ is immersed into $\PP^n$.

\medskip
The theory of deformation and moduli of aCM subschemes has been widely studied, the first results in this subject going back to Hilbert, we refer to \cite[Chapter 8]{hart-def} for a review on the story of this subject.
An important result is the main theorem of \cite{elling} stating that aCM subschemes of codimension $2$ form an open subset of the Hilbert scheme. This result will be useful for us to study the existence of foliations on $\PP^3$ with a given curve as its singular set. 

\medskip
Finally, we would like to state the following property relative to aCM varieties: from \cite[Proposition 8.6, p.~63]{hart-def} we know that a subvariety $Y\subset \PP^n$ is aCM if and only if 
\begin{equation}\label{hart-acm}
H^i_*(\mathcal{I}_Y)=0, \text{ for } 1\leq i \leq dim (Y),
\end{equation}
where we are using the notation $H^i_*(\F) := \bigoplus_{\ell\in\mathbb{Z}} H^i(\PP^n,\F(\ell))$ and we are denoting as $\mathcal{I}_Y$ the sheaf of ideals associated to the variety $Y$.

\section{Main Theorem}

Along this section we will give our main result, which is that $\KK$ is aCM under suitable hypotheses. Before doing that, we will prove, first that the unfolding ideal $I(\omega)$ is saturated and then, that it is always aCM.

We recall the reader that we will always consider $\PP^n$ with $n\geq 3$ and that, unless we mentioned it, all cohomologies are going to be on $\PP^n$.

\medskip

\begin{proposition}\label{propI}
Let $\omega\in\FF(\PP^n,e)$, then the first order unfolding ideal $I(\omega)$ is saturated.
\end{proposition}
\begin{proof}
Let $h\in S$ be such that there is an $N\in \NN$ such that $x_i^Nh\in I$ (for each $i=0,\dots,n$). 
Then, by definition, there is for each $0\leq i\leq n$ a form $\eta_i\in \Omega^1_S$ such that $x_i^N h d\omega=\omega\wedge \eta_i$.
Let $\tilde{\eta}_i\in \Omega^1_{S_{x_i}}$ be defined by $\tilde{\eta}_i:=\frac{1}{x_i^N}\eta_i$. 
We can view $\tilde{\eta}_i$ as a section of the sheafification $\widetilde{\Omega^1_S}$ of $\Omega^1_S$, defined on the open set $U_i:=\{(x_0:\dots:x_n)\ |\ x_i\neq0\}$.
On $U_i\cap U_j$ we have $\omega\wedge (\tilde{\eta}_i-\tilde{\eta}_j)=0$.
As $\sing(\omega)$ does not include divisors, one has that the Koszul complex associated with $\omega$ has vanishing first homology; this implies that, on $U_i\cap U_j$, we have $[\tilde{\eta}_i]=[\tilde{\eta}_j]\in \widetilde{\Omega^1_S}/(\omega)$.
So $\{[\tilde{\eta}_i]\}_{i=0}^n$ defines a \v{C}ech $1$-cocycle on the sheaf $\widetilde{\Omega^1_S}/(\omega)$ for the affine covering $\{U_i\}_{i=0}^n$ of $\PP^n$.
As $\widetilde{\Omega^1_S}\cong \OO_{\PP^n}^{\oplus n+1}$ and $(\omega)\cong \OO_{\PP^n}(-e)$, from the short exact sequence
\[
0\to(\omega)\to \widetilde{\Omega^1_S}\to \widetilde{\Omega^1_S}/(\omega)\to 0
\]
we get the exact sequence of cohomology groups
\[
H^1\left(\OO_{\PP^n}(-e)\right)\to H^1\left(\OO_{\PP^n}^{\oplus n+1}\right)\to H^1\left(\widetilde{\Omega^1_S}/(\omega)\right)\to H^2\left( \OO_{\PP^n}(-e)\right)
\]

As $ H^1\left(\OO_{\PP^n}^{\oplus n+1}\right)=0$ if $n>1$ and $ H^2\left( \OO_{\PP^n}(-e)\right)=0$ if $n>2$ we have $H^1\left(\widetilde{\Omega^1_S}/(\omega)\right)=0$ and so there is an element $[\tilde{\eta}]\in H^0(\widetilde{\Omega^1_S}/(\omega))$.
So we get a form $\eta$ such that $h d \omega=\omega\wedge\eta$, then $h\in I$. 
\end{proof}

We now want to prove that the unfoldings ideal $I$ is aCM. For that, let us denote as $\II$ the sheafification of such ideal.

\begin{theorem}\label{teoI}
Let $\omega\in\FF(\PP^n,e)$, then $H^1_*(\II) = 0$. In particular, if $\omega\in\FF(\PP^3,e)$ then $\II$ is aCM.
\end{theorem}
\begin{proof}
From Proposition \ref{propI} we know that the ideal $I$ is saturated which implies that $H^0_*(\II)=I$.

As before, where are going to denote by $\widetilde{\Omega^1_S}$ the free sheaf associated to the free $S$-module $\Omega^1_S$. We are also going to consider the free $S$-module $\Omega^2_S$ and its sheafification $\widetilde{\Omega^2_S}$.

Lets consider the morphism
\[ 
\OO_{\PP^n} \oplus \widetilde{\Omega^1_S}/(\omega) \to \widetilde{\Omega^2_S}(e)
\]
that maps a local section $(h,\eta)$ as follows, 
\[
(h,\eta)\mapsto h d\omega - \omega\wedge\eta.
\]

Denoting by $\mathcal{G}$ the sheaf theoretic image of this morphism we get a short exact sequence,
\[
0\to \II \to \OO_{\PP^n} \oplus \widetilde{\Omega^1_S}/(\omega) \to \mathcal{G} \to 0.
\]

As $\widetilde{\Omega^1_S}/(\omega)$ is a free sheaf, we get that $H^1(\OO_{\PP^n}(k) \oplus \widetilde{\Omega^1_S}/(\omega)(k))=0$ for all $k\in \ZZ$. 
So to prove $H^1(I(k))=0$ it suffices to see that the morphism on global sections
\[
 H^0(\OO_{\PP^n}(k) \oplus \widetilde{\Omega^1_S}/(\omega)(k))\to H^0(\mathcal{G}(k)),
\]
is surjective.

Now lets denote $M(k):=H^0(\mathcal{G}(k))$. As $M(k)$ is a submodule of $\Omega^2_S(e+k)$ we can write elements of $M(k)$ as polynomial $2$-forms of degree $e+k$.
For a polynomial $2$-form $\theta$ to be in $M(k)$ it suffices to be in the localizations $M(k)_{(x_i)}$ for $i=0,\ldots,n$. 
From this we get that, in order for the morphism  $H^0(\OO_{\PP^n}(k) \oplus \widetilde{\Omega^1_S}/(\omega)(k))\to H^0(\mathcal{G}(k))$ to be surjective, we must see that for each $\theta\in \Omega^2_S(e+k)$ such that there is a $N>> 0$ such that for all $0\leq i \leq n$ there are pairs $(h_i,\eta_i)$, with $h_i$ a polynomial and $\eta_i\in \Omega^1_S/(\omega)(k)$, such that 
\[
x_i^N \theta=h_i d\omega - \omega\wedge \eta_i, 
\]
then there is a pair $(h,\eta)$ such that $\theta=h d\omega - \omega\wedge \eta$.

\

To prove what we need it suffices showing that if there is a 2-form $\theta\in\Omega^2_S$ such that $x_i\theta=h_i d\omega-\omega\wedge\eta_i$, then, there is polynomial $h'$ and a 1-form $\eta'$ such that
\[
 \theta = h' d\omega-\omega\wedge \eta'.
\]

For this, first we need to define the following sheaf of modules.

\begin{definition}
We define $D(\omega)$ as the kernel of the following map
\[
 \xymatrix@R=0pt@C=15pt{
H^0_*(T_{\PP^n}) \ar[r]^{} & \Omega^1_S \\
X \ar@{|->}[r]^{} & i_X(d\omega) \ .\\
}
\]
\end{definition}

\begin{remark}
\begin{itemize}
 \item[i)] Note that $D(\omega)$ it is a submodule of $H^0_*(T_\omega)$ since, given $X\in D(\omega)$ we have that $i_X(\omega\wedge d\omega) = i_X(\omega)d\omega=0$ which implies $i_X(\omega) = 0$.
 \item[ii)] ${D(\omega)}_\pp$ is free for every $\pp$ such that $d\omega_\pp\neq 0$.
\end{itemize}
\end{remark}

\begin{lemma}
Let $I+(\omega)\subseteq \Omega^1_S$ be the submodule of $\eta\in\Omega^1_S$ such that $h d\omega=\omega\wedge\eta$, and let $A$ be the cokernel of the sequence.
\[
0\to I+(\omega) \to \Omega^1_S\to A\to 0.
\]
Then we have $D(\omega)=A^\vee$.
\end{lemma}
\begin{proof}
We denote by $T_S$ the dual module to $\Omega^1_S$, which is a free module. An element $X\in A^\vee$ is an element $X\in T\simeq S^{n+1}$ such that the contraction of any elemnt of $I+(\omega)$ with $X$ is $0$. Then for all $h\in I \subseteq S$ we have $i_X(h d\omega)= i_X(\omega\wedge\eta)=0$, so $i_X(d\omega)=0$ and, in other words, $X\in D(\omega)$.
\smallskip
On the other hand, if $X\in D(\omega)$, then for all $\eta\in I+(\omega)\subseteq \Omega^1_S$ we have
\[
i_X(\eta)\omega=i_X(\omega\wedge\eta)=0.
\]
Hence $X\in \mathcal{A}^\vee$.
\smallskip
Being $\Omega^1_S$ reflexive and $I+(\omega)$ saturated in $\Omega^1_S$ (so $A$ is torsion free), we can apply \cite[Lemma 4.2]{quallbrunn} and conclude that $D(\omega)=A^\vee$ and $I+(\omega)= (T_S/D(\omega))^\vee$.
\end{proof}

\begin{lemma}
 Consider the morphism 
\[
  \xymatrix@R=0pt@C=15pt{
\Omega_S \ar[r]^{} & D(\omega)^\vee  \\
\eta\ar@{|->}[r]^{} & (X\mapsto i_X(\eta))\ .\\
}
\]
Then the kernel of this morphism is  $I+(\omega).$
\end{lemma}
\begin{proof}
Follows from the above lemma as this morphism factors through $\Omega^1_S\to A $. 
\end{proof}

\begin{remark}
 Given $\theta\in \Omega^2_S$ such that $x_i\theta=h_i d\omega - \omega\wedge \eta_i$ the form $\eta_i$ defines a morphism
 \[
\xymatrix@R=0pt@C=15pt{
D(\omega) \ar[r]^{[\eta_i]} & S \\
X \ar@{|->}[r]^{} & i_X(\eta_i) \ .\\
}
\]
Such morphism is the same that one gets by mapping $X\mapsto f$ where $f$ is such that $i_X(x_i\theta)=f\omega$. Therefore contraction with $\eta_i$ actually defines a morphism $D(\omega)\to (x_i)\cdot S$.
\end{remark}

\begin{lemma}
  Given $\theta\in \Omega^2_S$ such that $x_i\theta=h_i d\omega - \omega\wedge \eta_i$, let $[\eta_i]$ be the image of $\eta$ in $D(\omega)^\vee$. Then  $[\eta]\in (x_i)\cdot D(\omega)^\vee$.
\end{lemma}
\begin{proof}
By localizing in a prime $\pp$ such that ${D(\omega)}_\pp$ is free and such that  $(x_i)\subseteq\pp$ we have $\Hom({D(\omega)}_\pp, (x_i)_\pp)=(x_i)\cdot(D(\omega)^\vee)_\pp$ (because ${D(\omega)}_\pp$ is free). Then $[\eta]_\pp \in (x_i)\cdot(D(\omega)^\vee)_\pp$, hence $[\eta]\in (x_i)\cdot D(\omega)^\vee$.
\end{proof}

So, if there is a $\theta\in \Omega^2_S$ such that $x_i\theta=h_i d\omega - \omega\wedge \eta_i$ then $\eta\in (x_i)\cdot\Omega_S+(I+(\omega))$, so we can write $\eta= x_i\eta'+\eta''+ g\omega$ where $\eta''\in I+(\omega)\subseteq\Omega_S$.
Therefore 
\[
 h d\omega+\eta\wedge\omega=h d\omega+ x_i\eta'\wedge\omega+\eta''\wedge\omega=h d\omega+ x_i\eta'\wedge\omega+h''d\omega=(h-h'')d\omega+x_i\eta'\wedge\omega.
\]
From which it follows that 
\[
 (h-h'')d\omega=x_i(\theta-\eta'\wedge\omega)\Longrightarrow (h-h'')\in (x_i) \Longrightarrow \theta= k d\omega+\eta'\wedge\omega.
\]
This way we get that $H^1_*(\II) = 0$. Now, following eq. \ref{hart-acm} and Remark \ref{1notinI} we have that if $\omega\in\FF(\PP^3,e)$ then $\II$ is aCM.
\end{proof}

\

We now want to prove our main result:

\begin{maintheorem}\label{teo4}
Let $\omega\in\FF(\PP^3,e)$ such that $J=\sqrt{J}$ and $\KK\cap\LL =\emptyset$, then $\KK$ is aCM.
\end{maintheorem}
\begin{proof}
We need only to apply Theorem \ref{technical1}, so we know that $I=K$, and by our previous Theorem the result follows. 
\end{proof}

\section{Applications}
In order to apply our main result, let us remark some results of holomorphic vector bundles on the projective space. The main reference is the book \cite{okonek}.

We say that a vector bundle $E$ over $\PP^n$ splits if it is a direct sum of line bundles. 

The first result that we have in mind is the \emph{Splitting criterion of Horrocks} \cite[Theorem 2.3.1 p. 21]{okonek}.

\begin{theorem}[Horrocks]\label{spl}
 A holomorphic vector bundle $E$ over $\PP^n$ splits if and only if
\[
H^i_{\ast}(\PP^n,E)=0
\mbox{ for all }1\leq i\leq n-1.
\]
\end{theorem}

This criterion has a surprising consequence
\cite[Theorem 2.3.2 p. 22]{okonek}.

\begin{theorem}\label{spl1}
A holomorphic vector bundle $E$ over $\PP^n$ splits in a sum of line bundles when its restriction to some plane $\PP^2\subset\PP^n$ splits
\end{theorem}

For a a rank two vector bundle $E$ over $\PP^3$, we can say a little more \cite[Theorem 1.2 p. 1221]{omegar-marcio} see also \cite{gir-pancoll}.

\begin{theorem}\label{spl2} A rank two vector bundle $E$ over $\PP^3$ splits if and only if \linebreak $H_{\ast}^1(\PP^3,E)=0$ or $H^2_{\ast}(\PP^3,E)=0$
\end{theorem}
\begin{proof}
Since $E$ has rank two, $E^{\vee}=E(-c_1)$, wher $c_1$ denotes the  first Chern class of $E$.

From Serre duality and for any $k\in \ZZ$ we have 
\[
H^1(\PP^3,E(k))=H^2(\PP^3,E(k)^{\vee}\otimes\OO_{\PP^3}(-4))=H^2(\PP^3,E(-c_1-k-4)).
\]

It follows that  $H^1_{\ast}(\PP^3,E)=0$ if and only if $H^2_{\ast}(\PP^3,E)$. The splitting follows from Theorem \ref{spl}. 
\end{proof}

As a first application of Theorem \ref{teo4} we can state the following result, which can be related to one implication of \cite[Theorem 1, p.~2]{correa-etal}:

\begin{fappl}\label{SplitTan}
Let $\omega\in\FF(\PP^3,e)$ be a foliation such that $J=\sqrt{J}$ and $\LL=\emptyset$, where 
\[
\xymatrix{
0 \ar[r] & T_\omega \ar[r] & T_{\PP^n} \ar[r]^{\omega} & \mathscr{J}(e) \ar[r] & 0
}
\]
where $\mathscr{J}$ denotes the sheafification of the ideal $J$, then $T_\omega$ splits.
\end{fappl}

\begin{proof}
Since $\LL=\emptyset$ we have that $\mathscr{J}=\KK$, and so it is  aCM. On the other hand $\mathscr{J}=\KK$ implies that $T_{\omega}$ is locally free, see \cite[Theorem 3.2, p.~848]{gir-pancoll}. Now, we could apply \cite[Theorem 3.3, p.~849]{gir-pancoll} to get the splitting of $T_{\omega}$, but we choose to make the following computation for a sake of clarity.

So, we just consider the long exact sequence associated to the short exact sequence given in the statement
\[
\xymatrix@R=10pt@C=15pt{
0 \ar[r] & H^0(\PP^3,T_\omega(e')) \ar[r] & H^0(\PP^3,T_{\PP^n}(e')) \ar[r]^-{\omega} & H^0(\PP^3,\mathscr{J}(e'+e)) \ar[r]^-{\delta} & \\
\ar[r]^-{\delta} & H^1(\PP^3,T_\omega(e')) \ar[r] & \underset{\simeq 0}{H^1(\PP^3,T_{\PP^n}(e'))} \ar[r] & \underset{\simeq 0}{H^1(\PP^3,\mathscr{J}(e'+e))} \ar[r]^-{\delta} & \\
}
\]
\[
\xymatrix@R=10pt@C=15pt{
\ar[r]^-{\delta} & H^2(\PP^3,T_\omega(e')) \ar[r] & \underset{\simeq 0}{H^2(\PP^3,T_{\PP^n}(e'))} \ar[r] &  H^2(\PP^3,\mathscr{J}(e'+e)) \ar[r]^-{\delta} & \\
\ar[r]^-{\delta} & H^{3}(\PP^3,T_\omega(e')) \ar[r] & H^{3}(\PP^3,T_{\PP^n}(e')) \ar[r] & H^{3}(\PP^3,\mathscr{J}(e'+e))  &\\
}
\]
From where we see that
$H^2(\PP^3,T_\omega(e'))\simeq 0$, for all $e'$. By the Theorem \ref{spl2} $T_\omega$ splits, from where we get our result.
\end{proof}

\

We would like to mention here a geometric interpretation of our hypothesis of $J=\sqrt{J}$ and $\LL=\emptyset$, which can be seen as $J=\sqrt{J}$ and $\sing(\omega)=\KK=\overline{\KK_{set}}$, see \cite[Lemma 4.6, p.~12]{mmq}.
 
 \
 
Let us consider a connected component $\KK_0\subset \KK_{set}$. It is very well known, see  \cite{kupka} or \cite[Teorema 1.8, p.~38]{ln-book}, that there exits 
$\{(U_{\alpha},\varphi_{\alpha}),\eta\},$ where $\{U_{\alpha}\}$ is an open covering of $\KK_0$ by coordinate open sets of $\PP^n$, a family of submersions $\varphi_{\alpha}:U_{\alpha}\rightarrow\CC^2$ such that 
$\varphi_{\alpha}^{-1}(0)=U_{\alpha}\cap\KK_0$ and a germ of a 1--form $\eta\in \Omega^1_0(\CC^2)$ with an isolated singularity at $0$, called \textit{transversal type of} the component $\KK_0$, such that 
$\omega_{\alpha}=\varphi_{\alpha}^{\ast}\eta$ defines the foliation on the open set $U_{\alpha}$. It follows that $\omega_{\alpha}=\lambda_{\alpha\beta}\omega_{\beta}$ in $U_{\alpha}\cap U_{\beta}$ and $\lambda_{\alpha\beta}\in \OO_{\PP^n}^{\ast}(U_{\alpha}\cap U_{\beta})$ is the cocycle that defines the bundle $\OO_{\PP^n}(e)$ in $\PP^n=\cup U_{\alpha}$.

Therefore, a connected component of $\KK_0\subset\KK$ is the pair $(\KK_0,\eta_0)$ with $\eta_0$ its transversal type. $\KK_0$ is a codimension two submanifold. 

Now, if the transversal type $\eta_0=A(x,y)dx-B(x,y)dy$, the function \linebreak $(A(x,y),B(x,y))$ defines the scheme structure on the variety $\KK_0$. On the other hand, if $\varphi(\mathbf{z},x,y)=(x,y)$ at any point $p\in \KK_0$, the tangent sheaf $T_{\omega}(p) = \{X\in T_{\PP^n}(p):\ i_X\eta = 0 \}$ is locally free and generated as an $\OO_{\PP^n}$ module by the vector fields
\[
T_{\omega} = \left\langle X_{\omega},\frac{\partial}{\partial z_1},\dots,\frac{\partial}{\partial z_{n-2}}\right\rangle\quad\mbox{ where}\quad
X_{\omega} := B(x,y)\frac{\partial}{\partial x} + A(x,y)\frac{\partial}{\partial y} \ .
\]

The condition of $d\omega|_{\KK_{set}}\neq 0$ implies that 
\[
d\eta(0)=\left(\frac{\partial B}{\partial x}(0)+\frac{\partial A}{\partial y}(0)\right)dx\wedge dy\neq0
\]
and the linear part of $\eta$ has a non trivial eigenvalue. By making a linear change of coordinates we may assume that the eigenvalues are $1$ and $\lambda$, from where we get that 
\[
 \eta = (x + \{hot\}) \ dy - (\lambda y + \{hot\}) \ dx\ ,
\]
where $hot$ stands for {\lq}higher order terms{\rq}. In the case where $\lambda\neq 0$ it is clear that $\sing(\omega)$ is reduced around the component $\KK_0$. 

On the other hand, if $\lambda=0$, we say that the component $\KK_0$ has \textit{saddle--node} transversal type. In this case, there are $1\leq q\in\ZZ,\quad \alpha \in \CC$ and a holomorphic function $P$ with multiplicity at least $q$ at $(0,0)$, such that $\eta$ has the holomorphic normal form
\[
 \eta = \left(x(1+\alpha y^q)+yP(x,y)\right) \ dy -y^{q+1} \ dx 
\]
which forces the component $\KK_0$ to be non reduced. The Milnor number or multiplicity $\mu(\KK_0)$ at the component $\KK_0$ is $q$. 

Therefore, our hypothesis of $J=\sqrt{J}$ implies that there are no irreducible component of the Kupka set with saddle--node transversal type.
\
\begin{example} Consider the one parameter family of foliations 
\[\omega_t=(x+(1+t)y)z dx- xz dy-x(x+t y)dz\in \FF(\PP^3,3),\quad t\in\CC
\]

Set $J_t=J(\omega_t)$. For $t\neq 0$ we have $J_t=\sqrt{J_t}$, but it is not true for $t=0$.

On the other hand, the singular set is the Kupka scheme for all $t$, and the tangent sheaf $T_{\omega_t}=\OO_{\PP^3}\oplus\OO_{\PP^3}(1)$.
\end{example}

\
For the following example, we will use the next result of complex vector bundles.

\begin{lemma}\label{Lemma} A rank two holomorphic vector bundle $F$ over  $\PP^n$ with $c_1(F)=0\mbox{ or }  -1$ and $c_2(F)=0$ splits as $F\simeq \OO_{\PP^n}\oplus \OO_{\PP^n}(c_1(F))$.
\end{lemma}
\begin{proof} The proof in $\PP^2,$ may be found in an appendix of \cite{CCF} in the case $c_1(F)=0$. We include the proof for completness.

By Theorem \ref{spl1}, it is sufficient to prove the theorem in 
$\PP^2$.
By Riemann--Roch in $\PP^2,$ we have 
\[
\chi(F)=h^0(F)-h^1(F)+h^2(F)=\frac{1}{2}(c_1^2-2c_2+3c_1+4)
=\begin{cases}1 &\mbox{ if } c_1(F)=-1 \\ 
              2 &\mbox{ if } c_1(F)=0,
 \end{cases}             
\] 
in any case, we have  $h^0(F)+h^2(F)=\chi(F)+h^1(F)\geq 1.$

Since $F$ has rank two $F^{\vee}=F(-c_1(F))$, Serre duality implies
\[
h^2(F)=h^0(F(-(c_1(F)+3))),\quad\mbox{moreover}\quad c_1(F)+3>0
\]
 
Now  
$h^0(F)\geq h^0(F(-k))$ for all $k>0$, and we get $h^0(F)\geq1$.

A non zero section $\tau\in H^0(\PP^2,F)$, induces the exact sequence 
\begin{equation}\label{eq:exs}
 0\longrightarrow
    {\OO_{\PP^2}}\stackrel{\cdot\tau}{\longrightarrow} 
    F \longrightarrow \mathcal{Q}\longrightarrow 0\quad\mbox{with}\quad \mathcal{Q}= F/\OO_{\PP^2}
\end{equation}

The sequence (\ref{eq:exs}) is a free resolution of the torsion free sheaf $\mathcal{Q}$. Therefore, $\mathcal{Q}\simeq \mathcal{J}_{\Sigma}$ an ideal sheaf up to twist, for some codimension two scheme $\Sigma$.

The second Chern class $c_2(F)=deg(\Sigma)=0$, we conclude that $\Sigma=\emptyset,$ and $\mathcal{Q}$ is locally free. Since $c_1(\mathcal{Q})=c_1(F)$ we get $\mathcal{Q}\simeq \OO_{\PP^2}(c_1(F))$. 

So that, the vector bundle $F$ is an extension by holomorphic line bundles, namely 
\[
0\longrightarrow \OO_{\PP^2}\stackrel{\cdot\tau}{\longrightarrow} F\rightarrow \OO_{\PP^2}(c_1(F))\longrightarrow 0.
\] 
Hence, following \cite[p.~15]{okonek}, $F$ splits and $F\simeq \OO_{\PP^2}\oplus\OO_{\PP^2}(c_1(F))$ as claimed.
\end{proof}

\

\begin{example} Consider in the affine space $\mathbb{A}_0=\{(1:x_1:x_2:x_3)\}\subset\PP^3$ the family of 1--forms
\[
\omega=x_3^2(\lambda_1 x_2 dx_1+\lambda_2 x_1 dx_2)+x_1x_2(1+\lambda_3x_3)dx_3,\qquad \lambda_1\lambda_2\lambda_3(\lambda_1-\lambda_2)\neq 0.
\]

Let $\pi(y_0:y_1:y_2:y_3)=(x_1,x_2,x_3),\quad x_i=y_i/y_0$ with $i=1,2,3.$ Then 
\[
\Omega=y_0^5\pi^{\ast}\omega= A_0(y)dy_0+A_1(y)dy_1+A_2(y)dy_2+A_3(y)dy_3,
\]
where
\[
\begin{array}{ccl}
A_0(y_0,y_1,y_2,y_3)& = & \lambda_0y_1y_2y_3^2-y_0y_1y_2y_3,\quad \lambda_0=-(
\lambda_1+\lambda_2+\lambda_3) \\
A_1(y_0,y_1,y_2,y_3)& = & \lambda_1 y_0y_2y_3^2  \\
A_2(y_0,y_1,y_2,y_3)& = & \lambda_2 y_0y_1y_3^2  \\
A_3(y_0,y_1,y_2,y_3)& = & y_0^2y_1y_2+\lambda_3 y_0y_1y_2y_3
\end{array}
\]

We recall from \cite[Theorem 3.1]{Jardim-etal} the following formulas
\begin{align*}
c_1(T_\Omega) &= 2 - d\\
 c_2(T_\Omega) &= d^2 + 2 - degree(S_2(\Omega))
\end{align*}
where $d$ denotes the degree of the foliation defined by $\Omega$.

\

We consider two cases:

If $\lambda_0=0$ the 1--form $\Omega$ is singular along $y_0=0$. After dividing by $y_0$ we get $\Omega'=\Omega/y_0\in H^0(\Omega_{\PP^3}^1(4))$
which has the following properties.

The singular set are the 4 lines of Kupka type and its multiplicity $\mu$ as listed below
\[
\begin{array}{lll}
&(K_{03}=\{y_0=y_3=0\},(y_0+\lambda_3 y_3)dy_3-y_3dy_0)\quad &\mu(K_{03})=1 \\
&(K_{12}=\{y_1=y_2=0\},\lambda_1y_2dy_1+\lambda_2y_1dy_2)\quad &\mu(K_{12})=1 \\
&(K_{13}=\{y_1=y_3=0\},\lambda_1y_3^2dy_1+y_1(1+\lambda_3 y_3)dy_3)\quad &\mu(K_{13})=2 \\
&(K_{23}=\{y_2=y_3=0\},\lambda_2y_3^2dy_2+y_2(1+\lambda_3 y_3)dy_3)\quad &\mu(K_{23})=2
\end{array}
\]

The singular set has pure codimension two and degree 6. The foliation 
has degree 2, therefore the tangent sheaf $T_{\Omega'}$ is locally free, see Remark \ref{rem-1},  with Chern classes $(c_1(T_{\Omega}),c_2(T_{\Omega}))=(0,0)$ then, by \cite[Theorem 2.3.2, p.22]{okonek} and by Lemma \ref{Lemma},  $T_{\Omega'}=\OO_{\PP^3}\oplus\OO_{\PP^3}$.

On the other hand, if $\lambda_0\neq0$, the form $\Omega\in H^0(\Omega_{\PP^3}^1(5))$ and  its singular set has 6 components 
$\{y_i=y_j=0\}_{0\leq i<j\leq 3}$ and 5 of them, with the exception of the line $L_{03}=\{y_0=y_3=0\}$, are of Kupka type.
\[
\begin{array}{lll}
&(K_{01}=\{y_0=y_1=0\},y_1(\lambda_0-y_0)dy_0+\lambda_1y_0dy_1)\quad 
&\mu(K_{01})=1 \\
&(K_{02}=\{y_0=y_2=0\},y_2(\lambda_0-y_0)dy_0+\lambda_2y_0dy_2)\quad
&\mu(K_{02})=1 \\
&(L_{03}=\{y_0=y_3=0\},(\lambda_0 y_3^2-y_0y_3)dy_0+(\lambda_3y_0y_3+y_0^2)dy_3)\quad &\mu(L_{03})=4 \\
&(K_{12}=\{y_1=y_2=0\},\lambda_1y_2dy_1+\lambda_2y_1dy_2)\quad &\mu(K_{12})=1 \\
&(K_{13}=\{y_1=y_3=0\},\lambda_1y_3^2dy_1+y_1(1+\lambda_3 y_3)dy_3)\quad &\mu(K_{13})=2 \\
&(K_{23}=\{y_2=y_3=0\},\lambda_2y_3^2dy_2+y_2(1+\lambda_3 y_3)dy_3)\quad &\mu(K_{23})=2
\end{array}
\]

Again, by the Remark \ref{rem-1}, $T_{\Omega}$ is locally free and has Chern classes \[(c_1(T_{\Omega}),c_2(T_{\Omega}))=(-1,0).\] Therefore by \cite[Theorem 2.3.2, p.22]{okonek} and the Lemma \ref{Lemma}, we get that
 $T_{\Omega}\simeq \OO_{\PP^3}\oplus\OO_{\PP^3}(-1)$
\end{example}
\
\begin{remark}\label{rem-1}
In \cite[Theorem 3.2, p.~848]{gir-pancoll} the authors show that, in $\PP^3$, $T_\omega$ is locally free $\iff$ $\sing(\omega)$ is a curve. Then, we can weaken our hypothesis in Theorem \ref{SplitTan}  by the following way: 

Let $\omega\in\FF(\PP^3,e)$ such that $J=\sqrt{J}$ and $\LL=\emptyset$, then $T_\omega$ splits.
\end{remark}

Using the formulation of the previous remark, we can give a positive answer to the question posed in \cite[Problem 2, p.~989]{omegar2} which says: `Is it true that $T_\omega$ splits for any generalized Kupka (GK) foliation $\omega$ on $\PP^3$?'

For that, let us first recall what the authors mean by a GK foliation, from \cite[Definition 1, p.~988]{omegar2} we have that: Let $\omega$ be an integrable 1-form defined in a neighborhood of $p\in\CC^3$. We say that $p$ is a GK singularity of $\omega$ if $\omega(p) = 0$ and either $d\omega(p)\neq 0$ or $p$ is an isolated zero of $d\omega$.

By Malgrange's result \cite[Th\'eor\`eme (0.1), p.~163]{m-f1} we know that GK points can be found only in codimension 2 components of the singular locus of $\omega$, as points in the closure of the set $\{p\in\PP^n\ : \ \omega(p)=0 \text{ and } d\omega(p)\neq 0 \}$ which coincides with $\KK$, under the hypothesis that we are considering $J=\sqrt{J}$. This way, asking $\omega$ to be GK is the same as asking $\omega$ to have $\LL=\emptyset$ and, together with Remark \ref{rem-1}, we have the question answered. This allow us to conclude the following theorem:

\begin{gk}\label{gk-sing}
Let $\omega\in\FF(\PP^3,e)$ be a foliation with GK singularities and with $K=\sqrt{K}$ or equivalently without components with saddle--node transversal type. Then, the tangent sheaf is locally free and splits. 
\end{gk}
\
\begin{remark}
\begin{enumerate}
\item Using the same argument as before, we can generalize \linebreak\cite[Lemma 2, p.~1009]{omegar2}, in the following way: Let $\omega\in\FF(\PP^n,e)$ and suppose that there exists a 3-plane $E$ such that $\omega|_{E}$ is a GK foliation without saddle--node transversal type, then $T_\omega|_E$ splits. It follows that $\omega=\pi^{\ast}\omega|_E$, where $\pi:\PP^n\dashrightarrow \PP^3$ is a linear projection. 

\item Since the tangent sheaf splits, in particular $H^1(T_\omega)\simeq H^2(T_{\omega})=0$. It follows from \cite[equation 0.3, p. 52]{GM} that the space of the infinitesimal deformations of the foliation, coincides with the space of deformations of its transversal structure.   
\end{enumerate}
\end{remark}

\

Another consequence of the Theorem \ref{teo4}, is given by the following property related to the connectedness of $\KK$. With this statement we give a partial answer to the question posted in \cite[2.3. Lieu singulier (2), p.~657]{cerveau}, which says: `Let $\omega\in\FF(\PP^3,e)$. Is the union of the irreducible components of codimension 2 of $\sing(\omega)$ connected?'

\begin{sappl}\label{K-connected}
 Let $\omega\in\mathcal{U}\subset\FF(\PP^n,e)$, then $\KK$ is connected.
\end{sappl}
\begin{proof}
By considering the short exact sequence
\[
\xymatrix{
0 \ar[r] & \II \ar[r]^{i} & \OO_{\PP^n} \ar[r]^{\pi} & \OO_{\PP^n}/\II \ar[r] & 0
}
\]
and its exact sequence of cohomology
\[
\xymatrix{
0 \ar[r] & H^0(\II) \ar[r]^{i} & H^0(\OO_{\PP^n}) \ar[r]^{\pi} & H^0(\OO_{\PP^n}/\II) \ar[r] & \underset{\simeq 0}{H^1(\II)}
}
\]
from where we have that the surjectivity of the application $\pi:H^0(\OO_{\PP^n}) \to H^0(\OO_{\PP^n}/\II)$ implies the connectedness of $Proj(\OO_{\PP^n}/\II)$.

By our hypothesis of $\omega\in\mathcal{U}$ we have that $\sqrt{I}=\sqrt{K}$, by Theorem \ref{teo1}, then if $Proj(S/I)\simeq Proj(\OO_{\PP^n}/\II)$ is connected then $Proj(S/K)= \KK$ is connected. 
\end{proof}

\begin{remark}
By the above corollary one can remove the hypothesis of connectedness in \cite[Theorem 3.3 and 3.4, pp.~1226-1227]{omegar-marcio}, since the transversal type can not be saddle--node.
\end{remark}

\bigskip

Regarding the next question posed in \emph{loc. cit.}, see \cite[2.4. Lieu singulier (3), p.~657]{cerveau}, which says: `Let $\gamma$ be an irreducible curve in $\PP^3$. Is there a foliation $\omega$ such that $\sing(\omega)=\gamma$?', we can say the following.

\medskip
Let $\gamma\subset \PP^3$ be a reduced and irreducible curve.
If there is a foliation $\omega$ such that $\sing(\omega)=\gamma$ then in particular $\sing(\omega)$ is reduced, so $\KK\neq\emptyset$ and $\sing(\omega)=\KK$.
Therefore $\gamma$ is aCM by Theorem \ref{teo4}.
The ideal of the curve $\gamma\subset\PP^3$ is then generated by the coefficients of $\omega=\sum_{i=0}^3f_i(x)dx_i$. So we have an aCM ideal with $4$ generators $I(\gamma)=(f_0,f_1,f_2,f_3)$.
Suppose this presentation is redundant, so there is a non-trivial linear combination $\sum a_if_i(x)=0$ with $a_i\in\CC$.
Then we have a (twisted) vector field $X=\sum a_i\frac{\partial}{\partial x_i}\in H^0(\PP^3,T\PP^3(-1))$ such that $i_X \omega=0$.
This implies $\omega$ is the pull-back of a foliation on $\PP^2$ under a linear projection $\PP^3\to\PP^2$.
So $\sing(\omega)$ is a union of lines, contradicting the irreducibility of $\gamma$.
Then we have a non-redundant presentation of $I(\gamma)$, as $\gamma$ is aCM of codimension $2$, by \cite[Proposition 8.7, p.~63]{hart-def} this must be of the following form:
\[
0\to \OO_{\PP^3}(a_0)\oplus\OO_{\PP^3}(a_1)\oplus\OO_{\PP^3}(a_2)\to\OO_{\PP^3}(1-e)^{\oplus 4} \to I(\gamma)\to 0.
\]

The projectivity of $\omega$ gives $i_R \omega=0$, this in turn gives a non-trivial relation $\sum_{i=0}^3x_if_i=0$ of degree $1$. 
So $\gamma$ necessarily must be an aCM curve with minimal presentation
\[
0\to \OO_{\PP^3}(-e)\oplus\OO_{\PP^3}(a_1)\oplus\OO_{\PP^3}(a_2)\to\OO_{\PP^3}(1-e)^{\oplus 4} \to I(\gamma)\to 0.
\]
We recover as a particular case the example of \cite[1.4 An example, p.~101]{cerveau-linsneto}, stating that the twisted cubic cannot be the singular set of a foliation, as the minimal presentation of this curve has three generators with two relations.

\medskip
Conversely, consider an aCM curve with presentation
\[
0\to \OO_{\PP^3}(-e)\oplus\OO_{\PP^3}(a_1)\oplus\OO_{\PP^3}(a_2)\to\OO_{\PP^3}(1-e)^{\oplus 4} \to I(\gamma)\to 0.
\]
Suppose we have a relation of degree $1$ that is generic in the sense that, writing it as $\sum_{i=0}^3 l_i(x)f_i(x)=0$ with $l_i$, $(i=0,\dots,3)$ linear forms, the map $x\mapsto (l_0(x):\dots:l_3(x))$ is an isomorphism of $\PP^3$.
We can take $y_i=l_i(x)$ as coordinates and construct the form $\sum f(y)dy_i$.
This degree $e$ form will be projective as $\sum y_i\frac{\partial}{\partial y_i}$ annihilates $\omega$.
The distribution defined by this $\omega$ is split and generated by the two twisted vector fields $X_1=\sum b_i(y)\frac{\partial}{\partial y_i}$ and $X_2=\sum c_i(y)\frac{\partial}{\partial y_i}$, where $\sum b_if_i=0$ and $\sum c_i f_i=0$ are the other two relations in the presentation of $I(\gamma)$.
However, this distribution may not be involutive.

\medskip
By the main theorem of \cite[Th\'eor\`eme 1, p.~424]{elling} there is an open subset of the Hilbert scheme parameterizing aCM curves in $\PP^3$ whose homogeneous ideal is generated by four polynomials of degree $e-1$ with three relations of degree $1$, $a_1-e$ and $a_2-e$ respectively.
This open set is dominated by the rational variety whose points are $3\times 4$ matrices of homogeneous polynomials with columns of degree $1$, $a_1-e$ and $a_2-e$ respectively, see \cite[Chapter 2, 8, p.~58]{hart-def}.
By restricting this open set further to one where we have a generic relation of degree $1$ we get a scheme parameterizing projective forms with split tangent distribution. On this open set we define a closed subscheme by  imposing the integrability of the form.

We therefore have: 
\begin{proposition}
There is a locally closed subscheme of the Hilbert scheme of aCM curves in $\PP^3$ with four generators of degree $e-1$, one relation of degree $1$, parameterizing curves $\gamma$ that are the singular set of a foliation defined by a form $\omega\in H^0(\PP^3,\Omega^1_{\PP^3}(e))$.
\end{proposition}

We can now give an alternative algebraic proof of the fact that if $\KK$ is compact then it is a complete intersection, see \cite[Theorem 1, p.~3]{omegar2016} or \cite[Proposition 2.4]{CCF}. The proof works in all cases.

\

Recall that if $\KK_0$ is a compact, connected component of the Kupka set of a foliation, if the normal bundle of $\KK_0$ has not vanishing first Chern class, then $\KK_0$ is reduced (the transversal type can not be a saddle node).  

\begin{TAppl}\label{complete-intersection}
 Let $\omega\in\FF(\PP^n,e)$ be such that $J=\sqrt{J}$, $\KK\cap\LL=\emptyset$ and $\KK=\KK_{set}$, then $\KK$ is a complete intersection.
\end{TAppl}
\begin{proof} 
Since $\KK=\KK_{set}$ the Kupka set is compact. Then following \cite[1.2, p.~5]{omegar2016} by a Serre construction, the normal bundle $N_\KK$ of $\KK$ in $\PP^n$, extends to a rank two holomorphic vector bundle $V$ of $\PP^n$, having a holomorphic section $\sigma$, vanishing on $\KK$, and defining the following exact sequence
 \[
 \xymatrix{
  0 \ar[r] & \OO_{\PP^n} \ar[r]^{\sigma} & V \ar[r] & \II_{\KK}(e) \ar[r] & 0\ ,
  }
 \]
where with $\II_\KK(e)$ we are denoting the ideal sheaf of the Kupka scheme defined by $\KK$.
Moreover $\KK$ is a complete intersection if and only if $V$ splits.

Such exact sequence, has a long exact sequence in cohomology given by
\[
\cdots\rightarrow H^i(\OO_{\PP^n}(e'))\rightarrow H^i(V(e'))
\rightarrow H^i(\II_{\KK}(e+e'))\rightarrow H^{i+1}(\OO_{\PP^n}(e'))\rightarrow\cdots
\]

If $n=3$, we see that $H^1(\PP^3,V(e'))=0$ for every $e'\in\ZZ$, by Theorem \ref{spl2}, the bundle $V$ splits and $\KK$ is a complete intersection.

For the general case, take a linear embedding $\ell:\PP^3\hookrightarrow \PP^n$ in a general position with respect to the foliation. 

In this case, $\ell^{\ast}\omega\in\FF(\PP^3,e)$ and $\ell^{-1}\KK=\KK_{\ell^{\ast}\omega}$. The vector bundle 
$\ell^{\ast}V=V|_{\PP^3}$ splits and then $V$ also splits, then $\KK$ is a complete intersection.
\end{proof}

\section{Obstructions for integrability}\label{aci}

Along this section we will state two results that give obstructions to the condition of integrability of $\omega\in \mathscr{D}(\PP^n,e)$.

\begin{lemma}
 Let $\omega\in H^0(\Omega^1_{\PP^n}(e))$, then $I(\omega)\neq 0$ if and only if $\omega$ is integrable. And, in case $\omega$ is integrable, we have $J(\omega)\subset I(\omega)$.
\end{lemma}
 \begin{proof}
  If $\omega$ is integrable then we can contract the equation $\omega\wedge d\omega=0$ by some vector field $X$ and we get
 \[
  i_X\omega \ d\omega = \omega\wedge i_Xd\omega,
 \]
showing that $0\neq J(\omega)\subset I(\omega)$. To see that the other implication holds, we must see that if $I(\omega)\neq 0$ then $\omega$ is integrable, this can be seen by multiplying by $\omega$ the equation
 \[
  h\ d\omega =\omega\wedge \eta
 \]
which shows that $\omega\wedge d\omega=0$.
 \end{proof}

We can state the following result, which gives an algebraic criterion for the integrability of $\omega$:

\begin{Unfold}\label{teo5}
 Let $\omega\in H^0(\Omega^1_{\PP^n}(e))$, then $J(\omega)\subset I(\omega) \iff \omega\wedge d\omega =0$.
\end{Unfold}
\begin{proof}
We just need to observe that we always have $0\neq J(\omega)$, then the result is immediate from the previous Lemma.
 \end{proof}

\bigskip

In this way, we have that the integrability condition on $\omega$, $\omega\wedge d\omega=0$, can be seen algebraically, also in terms of the short exact sequence
\[
 \xymatrix@R=10pt{
 0 \ar[r] & J(\omega) \ar[r] & I(\omega) \ar[r] & I(\omega)/J(\omega) \ar[r] & 0}
\]

This short exact sequence have a long sequence of cohomology associated to it given by, after sheafification, 
\begin{equation}\label{longexact1}
 \xymatrix@R=10pt@C=10pt{
 0 \ar[r] & H^0(\mathscr{J}(t)) \ar[r] & H^0(\II(t)) \ar[r] & H^0((\II/\mathscr{J})(t)) \ar[r]^-{\delta_1} & H^1(\mathscr{J}(t)) \ar[r] &\\ 
 \ar[r] & H^1(\II(t)) \ar[r] & H^1((\II/\mathscr{J})(t)) \ar[r]^-{\delta_2} & H^2(\mathscr{J}(t)) \ar[r] & H^2(\II(t)) \ar[r] & \ldots &  }
\end{equation}

Then, from Theorem \ref{teo4} we have:
\begin{corollary}
Let $\omega\in\FF(\PP^n,e)$ be as in Theorem \ref{teo4} then we have that $\delta_1$ is an epimorphism.
\end{corollary}
\begin{proof}
The annihilation of $H^1(\II(t))=H^1(\KK(t))=0$ is given by the morphism $\delta_1$ being an epimorphism.
\end{proof}

\

The following result was obtained in dimension 3 in \cite[Theorem 3.10, p.~12]{Jardim-etal}. We present here a shorter and almost completely self contained proof that works for any dimension $n\geq3$.

\begin{FoAppl}\label{teo6}
Let us consider $\omega\in H^0(\Omega^1_{\PP^n}(e))$ such that $\sing(\omega)$ is a smooth variety and such that $\sing(\omega)=S_2(\omega)$. Then $\omega$ is defined by a degree 0 rational foliation or it is not integrable.
 \end{FoAppl}
\begin{proof}
 Let us suppose that $\omega$ it is integrable. Since $\sing(\omega)$ is smooth variety, then it is reduced, and then by \cite[Theorem 4.24, p.~18]{mmq} we know that $\KK\neq\emptyset$, which implies $\KK=\sing(\omega)=S_2(\omega)$. Now by Corollary \ref{complete-intersection} we know that $\KK$ is complete intersection, and thus, by \cite[1.3.4 Theorem A, p.~99]{cerveau-linsneto} it is defined by a rational foliation of the type $(r,s)$, \emph{i.e.}, $\omega$ of type $\omega= r f dg - s g df$ for $f$ and $g$ polynomials of degree $r$ and $s$ respectively with $r+s=e$. Now, by \cite[Theorem 3, (i), p.~133]{fmi}, we know that $\sing(\omega)=S_2(\omega)$ if and only if $r=s=1$, what concludes our proof.
\end{proof}


\

\begin{tabular}{l l}
Omegar Calvo--Andrade$^*$ \hspace{3cm}\null&\textsf{jose.calvo@cimat.mx}\\
Ariel Molinuevo$^*$  &\textsf{ariel.molinuevo@cimat.mx}\\
Federico Quallbrunn$^\dagger$  &\textsf{fquallb@dm.uba.ar}\\
&\\
$^*$\textsc{Centro de investigaci\'on } & \\
\textsc{en Matem\'aticas, A.C.}&\\
\textsc{Callejon : Jalisco s/n} &\\
\textsc{Colonia : Valenciana} &\\
\textsc{CP 36240} &\\
\textsc{Guanajuato, Gto.}&\\
\textsc{M\'exico}&\\
&   \\
$^\dagger$\textsc{Departamento de Matem\'atica}&\\
\textsc{Pabell\'on I}&\\
\textsc{Ciudad Universitaria}&\\
\textsc{CP C1428EGA}&\\
\textsc{Buenos Aires}&\\
\textsc{Argentina}&\\
\end{tabular}

\end{document}